\documentclass[twoside,11pt]{article}
\usepackage{amsmath,amssymb,fdsymbol}
\usepackage{url}
\renewcommand{\Re}{\text{Re\,}}
\renewcommand{\Im}{\text{Im\,}}

\begin{document}
\newtheorem{theorem}{Theorem}
\newtheorem{lemma}[theorem]{Lemma}
\newtheorem{corollary}[theorem]{Corollary}
\newtheorem{definition}[theorem]{Definition}
\newtheorem{example}[theorem]{Example}
\pagenumbering{roman}
\renewcommand{\thetheorem}{\thesection.\arabic{theorem}}
\renewcommand{\thelemma}{\thesection.\arabic{lemma}}
\newenvironment{proof}{\noindent{\bf{Proof.\/}}}{\hfill$\blacksquare$\vskip0.1in}
\renewcommand{\thetable}{\thesection.\arabic{table}}
\renewcommand{\thedefinition}{\thesection.\arabic{definition}}
\renewcommand{\theexample}{\thesection.\arabic{example}}
\renewcommand{\theequation}{\thesection.\arabic{equation}}
\newcommand{\mysection}[1]{\section{#1}\setcounter{equation}{0}
\setcounter{theorem}{0} \setcounter{lemma}{0}
\setcounter{definition}{0}}
\newcommand{\mrm}{\mathrm}
\newcommand{\be}{\begin{equation}}
\newcommand{\ee}{\end{equation}}

\newcommand{\ben}{\begin{enumerate}}
\newcommand{\een}{\end{enumerate}}

\title
{\bf Unified  Compact Numerical Quadrature Formulas for Hadamard Finite Parts of Singular Integrals of Periodic Functions}

\author
{Avram Sidi\\
Computer Science Department\\
Technion - Israel Institute of Technology\\ Haifa 32000, Israel\\
E-mail:\quad  \url{asidi@cs.technion.ac.il}\\
URL:\quad    \url{http://www.cs.technion.ac.il/~asidi}}
\bigskip\bigskip
\maketitle \thispagestyle{empty}
\newpage\noindent

\begin{abstract}
 We consider the numerical computation of   finite-range singular
  integrals
  $$I[f]=\intBar^b_a f(x)\,dx,\quad f(x)=\frac{g(x)}{(x-t)^m},\quad m=1,2,\ldots,\quad
  a<t<b,$$  that are defined in the sense of Hadamard Finite Part, assuming that $g\in C^\infty[a,b]$ and $f(x)\in C^\infty(\mathbb{R}_t)$ is $T$-periodic with
  $\mathbb{R}_t=\mathbb{R}\setminus\{t+ kT\}^\infty_{k=-\infty}$,
  $T=b-a$. Using a generalization of the Euler--Maclaurin expansion developed in [A. Sidi,
  {Euler--Maclaurin} expansions for integrals with arbitrary algebraic   endpoint singularities.
{\em Math. Comp.}, 81:2159--2173, 2012],  we unify the treatment of these integrals. For each $m$, we  develop a number of  numerical quadrature formulas
$\widehat{T}^{(s)}_{m,n}[f]$ of trapezoidal type for $I[f]$.
For example, three numerical quadrature formulas of trapezoidal type result from this approach for the case $m=3$, and these are
   \begin{align*} \widehat{T}^{(0)}_{3,n}[f]&=h\sum^{n-1}_{j=1}f(t+jh)-\frac{\pi^2}{3}\,g'(t)\,h^{-1}
   +\frac{1}{6}\,g'''(t)\,h, \quad h=\frac{T}{n},\\
 \widehat{T}^{(1)}_{3,n}[f]&=h\sum^n_{j=1}f(t+jh-h/2)-\pi^2\,g'(t)\,h^{-1},\quad h=\frac{T}{n},\\
 \widehat{T}^{(2)}_{3,n}[f]&=2h\sum^n_{j=1}f(t+jh-h/2)-
\frac{h}{2}\sum^{2n}_{j=1}f(t+jh/2-h/4),\quad h=\frac{T}{n}.\end{align*}
   For all $m$ and $s$, we   show that all of the numerical quadrature formulas  $\widehat{T}^{(s)}_{m,n}[f]$ have    spectral accuracy; that is,
  $$  \widehat{T}^{(s)}_{m,n}[f]-I[f]=o(n^{-\mu})\quad\text{as $n\to\infty$}\quad \forall \mu>0.$$
   We  provide a  numerical example involving a periodic integrand with $m=3$ that  confirms our convergence theory. We also show how the formulas $\widehat{T}{}^{(s)}_{3,n}[f]$ can be used
  in an efficient manner for solving supersingular integral  equations whose kernels have a $(x-t)^{-3}$ singularity. A similar approach can be applied  for all $m$.
 \end{abstract}

\vspace{1cm} \noindent {\bf Mathematics Subject Classification 2010:}
 41A55,  41A60, 45B05, 45E05, 65B15,  65D30, 65D32.

\vspace{1cm} \noindent {\bf Keywords and expressions:} Hadamard  Finite Part,
singular integrals, hypersingular integrals,
 supersingular integrals, generalized Euler--Maclaurin expansions, asymptotic expansions,  numerical quadrature, trapezoidal rule.

\thispagestyle{empty}
\newpage
\pagenumbering{arabic}

\section{Introduction and background} \label{se1}
  \setcounter{equation}{0} \setcounter{theorem}{0}

   In this work, we consider the efficient numerical computation of
    \be \label{eq1}
      I[f]=\intBar^b_a f(x)\,dx,\quad f(x)=\frac{g(x)}{(x-t)^m},
      \quad g\in C^\infty[a,b], \quad  m=1,2,\ldots,\quad
  a<t<b,\ee where
  \be \label{eq2}
  f(x)\ \ \text{is $T$-periodic},\quad
   f\in C^\infty(\mathbb{R}_t), \quad \mathbb{R}_t=\mathbb{R}\setminus\{t+kT\}^\infty_{-\infty},\quad  T=b-a.
  \ee
    Clearly, the integrals $\int^b_af(x)\,dx$  are {\em not} defined in the regular sense, but they {\em are} defined in the sense of {\em Hadamard Finite Part (HFP)},  the HFP of $\int^b_af(x)\,dx$ being  commonly denoted by $\intBar^b_a f(x)\,dx$.\footnote{\label{ft1}When $m=1$, the HFP of $\int^b_af(x)\,dx$ is also called its {\em Cauchy Principal Value (CPV)} and the accepted notation for it is $\intbar^b_a f(x)\,dx.$ When $m=2$,
  $\intBar^b_af(x)\,dx$ is called a {\em hypersingular integral,} and when $m=3$,
  $\intBar^b_af(x)\,dx$ is called a {\em supersingular integral.}\\
  We reserve the notation $\int^b_au(x)\,dx$ for integrals that exist in the regular sense.}

  By invoking a recent generalization of the Euler--Maclaurin (E--M) expansion developed in Sidi \cite[Theorem 2.3]{Sidi:2012:EME-P1} that also applies to both regular and  HFP integrals,
       we unify the treatments of the HFP integrals in \eqref{eq1}--\eqref{eq2} and derive a number of very effective numerical quadrature formulas for $I[f]$ for each $m\geq1$. In the process of derivation, we also obtain a result that shows that all the quadrature formulas derived here have {\em spectral} convergence. As examples, we provide the different quadrature formulas for the cases $m=1,2,3,4$  and illustrate the application  of those formulas with $m=3$ to a nontrivial numerical example.

  We note that the case $m=1$ was considered earlier in
  Sidi and Israeli \cite{Sidi:1988:QMP} and Sidi \cite{Sidi:2013:CNQ},
the technique used in \cite{Sidi:1988:QMP} being  different from that used in  \cite{Sidi:2013:CNQ}. The case $m=2$ was treated in \cite{Sidi:2013:CNQ}.
In  \cite{Sidi:2013:CNQ}, we also gave a detailed study of  the exactness and convergence properties of the numerical quadrature formulas for the cases with $m=1,2.$ In Sidi \cite{Sidi:2014:AES}, we considered further convergence properties of these formulas and,  in
 Sidi  \cite{Sidi:2014:RES},  we analyzed the numerical  stability issues related to the application of the Richardson extrapolation process to them. (For the  Richardson extrapolation process, see Sidi \cite[Chapters 1,2]{Sidi:2003:PEM}, for example.)

 For the definition and properties of
Hadamard Finite Part integrals, see the books by Davis and Rabinowitz \cite{Davis:1984:MNI}, Evans
\cite{Evans:1993:PNI},  Krommer and Ueberhuber \cite{Krommer:1998:CI}, and  Kythe and Sch{\"{a}}ferkotter \cite{Kythe:2005:HCM}, for
example. These integrals have most of the properties of regular integrals and some properties that are quite
unusual. For example, they are invariant with respect to translation, but they are not necessarily invariant
under a  scaling of the variable of integration, which is linear; therefore, they are not necessarily invariant
under a nonlinear   variable transformation either. Finally, $\intBar^b_a \phi(x)\,dx=\int^b_a \phi(x)\,dx$ when $\phi(x)$ is integrable over $[a,b]$ in the regular sense.
For more recent developments, see  the books  by Lifanov, Poltavskii, and Vainikko \cite{Lifanov:2004:HIE} and Ladopoulos \cite{Lifanov:2004:HIE}, for example. See also the papers by Kaya and Erdogan \cite{Kaya:1987:SIE},  Monegato  \cite{Monegato:1994:NEH}, \cite{Monegato:2009:DPA}. For an interesting  two-dimensional generalization, see
Lyness and Monegato \cite{Lyness:2005:AET}.

Cauchy principal value, hypersingular, and supersingular   integrals described in footnote$^{\ref{ft1}}$ arise in different branches of science and engineering,
such as fracture mechanics, elasticity, electromagnetic scattering, acoustics, and fluid mechanics, for example. They appear naturally in boundary integral equation formulations of boundary value problems in these disciplines.
  Periodic singular integrals arise naturally from Cauchy transforms
$\intBar_\Gamma \frac{w(\zeta)}{(\zeta-z)^m}\,d\zeta,$
where $\Gamma$ is an infinitely smooth  closed contour in the complex $z$-plane and
$z\in \Gamma$; we discuss this briefly in Section 5.

Various  numerical quadrature formulas for these integrals have been developed in several recent papers. Some of these papers, make use of  trapezoidal sums
or composite Simpson and Newton Cotes rules  with appropriate correction terms to account for the singularity at $x=t$; see Li and Sun \cite{Li:2010:NCR}, Li, Zhang, and Yu \cite{Li:2010:SUN},
Zeng, Li,  and  Huang \cite{Zeng:2014:NCQ}, and  Zhang, Wu, and Yu \cite{Zhang:2009:SCS}, for example.
The paper by  Huang, Wang, and Zhu \cite{Huang:2013:AEE} approaches the problem of computing HFP integrals of the form $\intBar^b_ag(x)/|x-t|^\beta\,dx$, (with the restriction $1<\beta\leq2$)
 by following  Sidi and Israeli \cite{Sidi:1988:QMP}, which is based on the generalizations  of the Euler--Maclaurin  expansion by Navot \cite{Navot:1961:EEM}, \cite{Navot:1962:FEE}.
 The papers   by Wu, Dai, and Zhang \cite{Wu:2010:SCM}  and
 by Wu and Sun \cite{Wu:2008:SNC} take similar  approaches. The approach of \cite{Sidi:2013:CNQ} is
  based on the most recent developments in Euler--Maclaurin expansions
 of \cite{Sidi:2012:EME-P1} that are valid for all HFP integrals even with possible {\em arbitrary} algebraic   endpoint singularities.

In the next section, we review   the author's generalization of the E--M expansion
for integrals  whose integrands are allowed to have arbitrary algebraic endpoint singularities. This generalization is given as Theorem \ref{th:1-2}.
In Section \ref{se3}, we apply Theorem \ref{th:1-2} to construct the
generalized E--M expansion
for $I[f]$  given in \eqref{eq1}--\eqref{eq2}.  In Section \ref{se4}, we develop a number of numerical quadrature formulas of trapezoidal type  for $I[f]$ with arbitrary $m$ and  analyze their convergence properties.
We also analyze their numerical stability in floating-point arithmetic.

When applied to the  HFP integrals $\intBar^b_a f(x)\,dx$ in \eqref{eq1}--\eqref{eq2}, all these  quadrature formulas possess the following favorable properties, which transpire from the developments in Sections \ref{se3} and \ref{se4}:
  \begin{enumerate}
  \item
  Unlike the quadrature formulas developed in the  papers mentioned above, they
  are {\em compact}  in that they consist of
  trapezoidal-like rules with very simple, yet sophisticated and unexpected,  ``correction'' terms to account for the singularity at $x=t$.
  \item They
  have a unified convergence theory that follows directly and very simply from the way they are derived.
  \item
  Unlike   the methods developed in  the papers mentioned above, which attain very limited accuracies, our methods
  enjoy {\em spectral} accuracy.
    \item
  Because they enjoy spectral accuracy, they are much more stable numerically than existing methods.
\end{enumerate}

In Section \ref{se5}, we apply the quadrature formulas for supersingular integrals ($m=3$) of Section \ref{se4} to a $T$-periodic $f(x)$ in $C^\infty(\mathbb{R}_t)$ and confirm numerically the convergence theory of Section \ref{se4}.
Finally, in Section \ref{se6}, we show how two of  these quadrature formulas, denoted  $\widehat{T}^{(0)}_{3,n}[\cdot]$ and  $\widehat{T}^{(2)}_{3,n}[\cdot]$,  can be used in the solution of supersingular integral equations.

Before proceeding to the next sections, we  would like to recall some of the properties of
the Riemann  Zeta function $\zeta(z)$ and the Bernoulli
 numbers $B_k$ and the connection between them for future reference:
\begin{gather}
B_0=1,\ \
B_1=-\frac{1}{2};\quad B_{2k+1}=0,\ \ B_{2k}\neq 0,\ \
k=1,2,\ldots,\nonumber \\
\zeta(0)=-\frac{1}{2};\quad \zeta(-2k)=0, \ \
\zeta(1-2k)=-\frac{B_{2k}}{2k}\neq 0,\ \ k=1,2,\ldots, \label{eq:2-4}\\ \zeta(2k)=(-1)^{k+1}\frac{(2\pi)^{2k}}{2(2k)!}B_{2k},\ \
k=1,2,\ldots \notag.
\end{gather}
For all these and much more, see Olver et al. \cite[Chapters 24, 25]{Olver:2010:NIST} or Luke \cite[Chapter 2]{Luke:1969:SFA1}, for example. See also Sidi \cite[Appendices D, E]{Sidi:2003:PEM}.

  \section{Generalization of the Euler--Maclaurin expansion to
  integrals with arbitrary algebraic  endpoint singularities}\label{se2}
   \setcounter{equation}{0} \setcounter{theorem}{0}
  The following theorem concerning the generalization of the E--M expansion to integrals with {\em arbitrary} algebraic endpoint singularities was published recently by Sidi  \cite[Theorem 2.3]{Sidi:2012:EME-P1}. It serves as the main analytical tool for all the developments in this paper.

\begin{theorem}\label{th:1-2}
Let $u\in C^{\infty}(a,b)$, and assume that $u(x)$ has the
asymptotic expansions
\be\label{eq37}\begin{split}
&u(x)\sim K(x-a)^{-1}+\sum^{\infty}_{s=0}c_s\,(x-a)^{\gamma_s}
\quad \text{as}\   x\to a+,\\
&u(x)\sim L(b-x)^{-1}+\sum^{\infty}_{s=0}d_s\,(b-x)^{\delta_s}
\quad \text{as}\  x\to b-,
\end{split}\ee
where the $\gamma_s$ and
$\delta_s$  are distinct complex numbers that satisfy
\begin{equation}\label{eq38}
\begin{matrix}
&\gamma_s\neq -1\quad \forall  s;\quad \text{\em Re\,}\gamma_0\leq\text{\em Re\,}\gamma_1\leq\text{\em Re\,}\gamma_2\leq\cdots;
&\lim_{s\to\infty}\text{\em Re\,}\gamma_s=+\infty,\\ \\
&\delta_s\neq -1\quad \forall s;\quad \text{\em Re\,}\delta_0\leq\text{\em Re\,}\delta_1\leq\text{\em Re\,}\delta_2\leq\cdots;
&\lim_{s\to\infty}\text{\em Re\,}\delta_s=+\infty.
\end{matrix}
\end{equation}
Assume furthermore that, for each positive integer $k$,
$u^{(k)}(x)$ has asymptotic expansions as $x\to a+$ and
$x\to b-$ that are obtained by differentiating those of
$u(x)$  term by term $k$ times.\footnote{We express this briefly by saying that
``the asymptotic expansions in \eqref{eq37} can be differentiated infinitely  many times.''}
Let also $h=(b-a)/n$ for  $n=1,2,\ldots\ .$ Then, as $h\to0$,
\begin{align}
h\sum^{n-1}_{j=1}u(a+jh)\sim\intBar^b_au(x)\,dx&+ K(C-\log h)+
\sum^{\infty}_{\substack{s=0\\ \gamma_s\not\in\{2,4,6,\ldots\}}}
c_s\,\zeta(-\gamma_s)\,h^{\gamma_s+1} \notag\\
&+L(C-\log h)+\sum^{\infty}_{\substack{s=0\\ \delta_s\not\in\{2,4,6,\ldots\}}}
d_s\,\zeta(-\delta_s)\,h^{\delta_s+1}, \label{eq39}
\end{align}
where $C=0.577\cdots$ is Euler's constant.\footnote{Note that the  constants $K$ and/or $L$ in \eqref{eq37} hence in  \eqref{eq39} can be zero.}
\end{theorem}

\noindent{\bf Remarks:}
\begin{enumerate}
\item
Note that if $K+L=0$ and $\Re\gamma_0>-1$ and $\Re\delta_0>-1$, then $\int^b_a u(x)\,dx$ exists as a regular integral; otherwise, it does not, but  its HFP does.
\item \label{re22}
When $u\in C^\infty[a,b]$, the Taylor series of $u(x)$ at $x=a$ and at $x=b$, whether convergent or divergent, are also (i)\,asymptotic expansions of $u(x)$ as $x\to a+$ and as $x\to b-$, respectively, and (ii)\,can be differentiated term-by-term any number of times.
Thus, Theorem \ref{th:1-2} applies without further assumptions on $u(x)$ when $u\in C^\infty[a,b]$.
\item
When $u\in C^\infty(a,b)$, the E--M expansion is completely determined by the asymptotic expansions of $u(x)$ as $x\to a+$ and as $x\to b-$, nothing  else being needed. What happens in $(a,b)$ is immaterial.
\item
It is clear from \eqref{eq39}  that the positive even integer powers of
$(x-a)$ and  $(b-x)$,
if present in the asymptotic expansions of $u(x)$ as
$x\to a+$ and $x\to b-$,
do not contribute to the asymptotic expansion of
$h\sum^{n-1}_{j=1}u(a+jh)$ as $h\to 0$, the reason being that
$\zeta(-2k)=0$ for $k=1,2,\ldots,$ by \eqref{eq:2-4}. We have included the ``limitations''
$\gamma_s\not\in\{2,4,6,\ldots\}$ and $\delta_s\not\in\{2,4,6,\ldots\}$ in the sums on the right-hand side of \eqref{eq39} only as ``reminders.''

\item Theorem \ref{th:1-2} is only a special case of a more general theorem
in \cite{Sidi:2012:EME-P1} involving the so-called ``offset trapezoidal rule''
$h\sum^{n-1}_{i=0}f(a+jh+\theta h)$, with $\theta\in[0,1]$ fixed,\footnote{Note that, with $\theta=1/2$, the offset trapezoidal rule becomes the mid-point rule.}
    that contains as special cases  all previously known generalizations of the E--M expansions
    for integrals with algebraic endpoint singularities. For  a further generalization
pertaining to arbitrary algebraic-logarithmic endpoint singularities, see
    Sidi \cite{Sidi:2012:EME-P2}.
\end{enumerate}

\section{Generalized  Euler--Maclaurin expansion for \\ $\intBar^b_ag(x)/(x-t)^m\,dx,\ m=1,2,\ldots$} \label{se3}
  \setcounter{equation}{0} \setcounter{theorem}{0}

  We now  present the derivation of the generalized E--M expansion  for the HFP integral  $I[f]$ in \eqref{eq1}--\eqref{eq2}.  As already mentioned,   our starting point and main   analytical  tool is Theorem \ref{th:1-2}.
Before we begin, we would like to mention that this has already been discussed in \cite{Sidi:2013:CNQ}, separately for even $m$ and odd $m$ and  using an indirect approach. Our approach here unifies the treatments for all $m$, is direct,  and  is much simpler than that in \cite{Sidi:2013:CNQ}.

First, we claim that, because $f(x)$ is $T$-periodic, with $T=b-a$, we can express $I[f]$ in \eqref{eq1} as
\be \label{eqint} I[f]=\intBar^{t+T}_t f(x)\,dx.\ee
As we are dealing with HFP integrals that are {\em not} defined in the regular sense, this
claim needs to be justified rigorously. For this, we need to recall some of the properties of HFP integrals we mentioned in Section \ref{se1}.
We begin by noting that
\be \label{eqintw} I[f]=\intBar^t_a f(x)\,dx+\intBar^b_t f(x)\,dx,\ee because HFP integrals are invariant with respect to the union of integration intervals. Next, we recall that HFP integrals are invariant under a translation of the interval of integration; therefore, under the variable transformation $y=x+T$, which is only a translation of the interval $[a,t]$ to $[b,t+T]$,  there holds
\be \label{eqintx}\intBar^t_a f(x)\,dx=\intBar^{t+T}_{a+T} f(y-T)\,dy=\intBar^{t+T}_{b} f(x-T)\,dx.\ee
Finally, by $T$-periodicity of $f(x)$, we have $f(x-T)=f(x)$, hence
\be \label{eqinty}\intBar^{t+T}_{b} f(x-T)\,dx=\intBar^{t+T}_{b} f(x)\,dx.\ee
The claim in \eqref{eqint} is now justified by combining
\eqref{eqintx} and \eqref{eqinty} in \eqref{eqintw}, thus obtaining
$$ I[f]=\intBar^{b}_{t} f(x)\,dx+\intBar^{t+T}_{b} f(x)\,dx
=\intBar^{t+T}_t f(x)\,dx.$$

With \eqref{eqint} justified, we now show that Theorem \ref{th:1-2} can be applied  {\em as is} to the integral
$\intBar^{t+T}_t f(x)\,dx$ {\em instead of} the integral $\intBar^{b}_a f(x)\,dx$. Of course, for this,  we need
to show that (i)\,$f(x)$ is infinitely differentiable on the interval $(t,t+T)$ and (ii)\,$f(x)$, as $x\to t+$ and as $x\to (t+T)-$, has  asymptotic expansions of
 the forms shown in Theorem \ref{th:1-2}. In doing so, we need to remember that neither $g(x)$ nor $(x-t)^{-m}$ is $T$-periodic even though $f(x)$ is.  The details follow.

 \begin{itemize}
\item
By the fact that $f\in C^\infty(\mathbb{R}_t)$  and  by $T$-periodicity
of $f(x)$, it is clear that $f\in C^\infty(t,t+T)$, with  singularities only at $x=t$ and $x=t+T$.
\item{\em  Asymptotic expansion of $f(x)$  as $x\to t+$:}\\
Expanding  $g(x)$ in a Taylor series at $x=t$, we obtain
$$ f(x)\sim \sum^\infty_{i=0}\frac{g^{(i)}(t)}{i!}\, (x-t)^{i-m} \quad \text{as $x\to t$},$$
which we write in the form
\be\label{eqf1} f(x)\sim\frac{g^{(m-1)}(t)}{(m-1)!}(x-t)^{-1}
+\sum^\infty_{\substack{i=0\\ i\neq m-1}}\frac{g^{(i)}(t)}{i!}\, (x-t)^{i-m} \quad \text{as $x\to t+$}.
\ee
\item {\em Asymptotic expansion of $f(x)$ as $x\to (t+T)-$:}\\
We first note that
 $$f(x)=f(x-T)=\frac{g(x-T)}{(x-T-t)^m}\quad \text{by $T$-periodicity of $f(x)$}.$$ Next,  expanding $g(x-T)$ in a Taylor series  at $x=t+T$, we obtain
$$ f(x)\sim \sum^\infty_{i=0}\frac{g^{(i)}(t)}{i!}\, (x-t-T)^{i-m} \quad \text{as $x\to (t+T)$},$$ which we write in the form
\begin{multline}\label{eqf2}
f(x)\sim -\frac{g^{(m-1)}(t)}{(m-1)!}(t+T-x)^{-1}\\
+\sum^\infty_{\substack{i=0\\ i\neq m-1}}(-1)^{i-m}\frac{g^{(i)}(t)}{i!}\, (t+T-x)^{i-m} \quad \text{as $x\to (t+T)-$}.
\end{multline}
\end{itemize}
Note that here we have recalled Remark \ref{re22}  concerning Taylor series expansions
 following the statement of Theorem \ref{th:1-2}.

Clearly,   Theorem \ref{th:1-2} applies with $a=t$ and $b=t+T$, and
$$ K=-L=\frac{g^{(m-1)}(t)}{(m-1)!}, \quad
\gamma_s=\delta_s=\begin{cases}s-m,&0\leq s\leq m-2\\ s+1-m, & s\geq m-1\end{cases}, $$
and
$$ c_s=\begin{cases}g^{(s)}(t)/s!,& 0\leq s\leq m-2\\
g^{(s+1)}(t)/(s+1)!,& s\geq m-1\end{cases}, \quad
d_s=\begin{cases}(-1)^{s-m}c_s,& 0\leq s\leq m-2,\\ (-1)^{s+1-m}c_s,&
s\geq m-1\end{cases}.$$
Letting  $h=T/n$,  and noting that the terms $K(x-t)^{-1}$ and $L(t+T-x)^{-1}$ in the  asymptotic expansions of $f(x)$ given in \eqref{eqf1} and \eqref{eqf2} make contributions that cancel each other for all $m$, we thus have the  asymptotic expansion
\be \label{eqEM1} h\sum_{j=1}^{n-1} f(t+jh)\sim I[f] +
\sum^\infty_{\substack{i=0\\ i\neq m-1}}[1+(-1)^{i-m}]
\frac{g^{(i)}(t)}{i!}\,\zeta(-i+m)\,h^{i-m+1} \quad \text{as $h\to0$}.\ee
Now, this asymptotic expansion assumes different forms depending on whether $m$ is even or odd. We actually have the following result:

\begin{theorem}\label{th11}
With $f(x)$ as in \eqref{eq1}--\eqref{eq2} and
\be\label{eqTtilde}\widetilde{T}_{m,n}[f]=h\sum^{n-1}_{j=1}f(t+jh),\quad h=T/n,\ee
 the following hold:
\begin{enumerate}
\item For $m$ even, $m=2r$, $r=1,2,\ldots,$
\be \label{evenm}\widetilde{T}_{2r,n}[f]=I[f]+
2\sum^r_{i=0}\frac{g^{(2i)}(t)}{(2i)!}\,\zeta(2r-2i)\,h^{-2r+2i+1}
+o(h^\mu)
\quad \text{as $n\to\infty$}\quad \forall \mu>0.\ee
\item For $m$ odd, $m=2r+1$, $r=0,1,\ldots,$
\be \label{oddm}\widetilde{T}_{2r+1,n}[f]=I[f]+
2\sum^r_{i=0}\frac{g^{(2i+1)}(t)}{(2i+1)!}\,\zeta(2r-2i)\,h^{-2r+2i+1}
+o(h^\mu)
\quad \text{as $n\to\infty$}\quad \forall \mu>0.\ee
\end{enumerate}
\end{theorem}
\begin{proof}
We consider the cases of even and odd $m$ separately.
\begin{enumerate}
\item
For $m=2r$, $r=1,2,\ldots,$ we have that only terms with even $i$ contribute to the infinite sum $\sum^\infty_{\substack{i=0\\ i\neq m-1}}$ in \eqref{eqEM1}, which reduces to
\be\label{eqty1} 2\sum^\infty_{i=0}\frac{g^{(2i)}(t)}{(2i)!}\,\zeta(2r-2i)\,h^{-2r+2i+1}.
\ee

\item
For $m=2r+1$, $r=0,1,\ldots,$ we have that only terms with odd $i$ contribute to the infinite sum $\sum^\infty_{\substack{i=0\\ i\neq m-1}}$ in \eqref{eqEM1}, which reduces to
\be\label{eqty2}
2\sum^\infty_{i=0}\frac{g^{(2i+1)}(t)}{(2i+1)!}\,\zeta(2r-2i)\,h^{-2r+2i+1}.
\ee
\end{enumerate}
Recalling that $\zeta(-2k)=0$ for $k=1,2, \ldots,$ we realize that all the terms with $i>r$ in the  two sums in \eqref{eqty1} and \eqref{eqty2} actually vanish.
This, of course, does {\em not} necessarily mean  that

$$\widetilde{T}_{2r,n}[f]=I[f]+
2\sum^r_{i=0}\frac{g^{(2i)}(t)}{(2i)!}\,\zeta(2r-2i)\,h^{-2r+2i+1},\quad r=1,2,\ldots,$$
$$\widetilde{T}_{2r+1,n}[f]=I[f]+
2\sum^r_{i=0}\frac{g^{(2i+1)}(t)}{(2i+!)!}\,\zeta(2r-2i)\,h^{-2r+2i+1},\quad r=0,1,\ldots.$$
Since there are no powers of $h$ in addition to $h^{-2r+1},h^{-2r+3},\ldots,h^{-1},h^1$ that are already present, a remainder term of order $o(h^\mu)$ for every $\mu>0$ is present on the right-hand side of each of these ``equalities.'' This completes the proof.
\end{proof}

As can be seen from \eqref{evenm} and \eqref{oddm},  the finite sums involving
$g(t)$ and its  derivatives are completely known  provided $g(x)$ and its  derivatives are known or can be computed, since $\zeta(0),\zeta(2),\ldots,\zeta(2r)$ are known from \eqref{eq:2-4}.
In the next section, we derive numerical quadrature formulas that rely on  (i)\,all of the $g^{(k)}(t)$, (ii)\,some of the $g^{(k)}(t)$, and (iii)\,none of the $g^{(k)}(t)$.

\section{Compact numerical quadrature formulas}
\label{se4}
\setcounter{equation}{0} \setcounter{theorem}{0}
\subsection{Development of numerical quadrature formulas}
Theorem \ref{th11} can be used to design numerical quadrature formulas in different ways. The first ones are obtained directly from \eqref{evenm} and \eqref{oddm}, and they read

\begin{align}\label{eqthat0even}
\widehat{T}^{(0)}_{m,n}[f]&=\widetilde{T}_{m,n}[f]-
2\sum^r_{i=0}\frac{g^{(2i)}(t)}{(2i)!}\,\zeta(2r-2i)\,h^{-2r+2i+1},\ \ m=2r,\ \ r=1,2,\ldots,\\
\label{eqthat0odd}
\widehat{T}^{(0)}_{m,n}[f]&=\widetilde{T}_{m,n}[f]-
2\sum^r_{i=0}\frac{g^{(2i+1)}(t)}{(2i+1)!}\,\zeta(2r-2i)\,h^{-2r+2i+1},
\ \ m=2r+1,\ \ r=0,1,\ldots.\end{align}

Clearly, $g(x)$ and  derivatives of $g(x)$ that are present in the asymptotic expansions of Theorem \ref{th11} are  an essential part of the formulas $\widehat{T}^{(0)}_{m,n}[f]$.
Numerical quadrature formulas that use less of this   information can be developed by applying a number of steps of a ``Richardson-like extrapolation'' process to the sequence
$\widehat{T}^{(0)}_{m,n}[f], \widehat{T}^{(0)}_{m,2n}[f], \widehat{T}^{(0)}_{m,4n}[f],\ldots,$ thereby eliminating
the powers of $h$ in the order $h^1,h^{-1},h^{-3},\ldots.$\footnote{Recall that, when applying the Richardson extrapolation process, we would eliminate the powers of $h$ in the order $h^{-2r+1},h^{-2r+3},\ldots,h^{-3},h^{-1},h^1$.}
For $m=1,2,3,4$, for example, we obtain the following quadrature formulas via this process:

\begin{enumerate}
\item
{\em The case $m=1$}:
\begin{align}
\widehat{T}^{(0)}_{1,n}[f]&=h\sum^{n-1}_{j=1}f(t+jh)+g'(t)h \label{eqT10}\\
\widehat{T}^{(1)}_{1,n}[f]&=h\sum^{n}_{j=1}f(t+jh-h/2) \label{eqT11} \end{align}
\item
{\em The case $m=2$}:
\begin{align}
\widehat{T}^{(0)}_{2,n}[f]&=h\sum^{n-1}_{j=1}f(t+jh)-\frac{\pi^2}{3}g(t)h^{-1}+
\frac{1}{2}g''(t)h \label{eqT20} \\
\widehat{T}^{(1)}_{2,n}[f]&=h\sum^{n}_{j=1}f(t+jh-h/2)-{\pi^2}g(t)h^{-1} \label{eqT21}\\
\widehat{T}^{(2)}_{2,n}[f]&=2h\sum^{n}_{j=1}f(t+jh-h/2)
-\frac{h}{2}\sum^{2n}_{j=1}f(t+jh/2-h/4) \label{eqT22}
\end{align}
\item
{\em The case $m=3$}:
\begin{align}
\widehat{T}^{(0)}_{3,n}[f]&=h\sum^{n-1}_{j=1}f(t+jh)-\frac{\pi^2}{3}g'(t)h^{-1}+
\frac{1}{6}g'''(t)h \label{eqT30} \\
\widehat{T}^{(1)}_{3,n}[f]&=h\sum^{n}_{j=1}f(t+jh-h/2)-{\pi^2}g'(t)h^{-1} \label{eqT31} \\
\widehat{T}^{(2)}_{3,n}[f]&=2h\sum^{n}_{j=1}f(t+jh-h/2)
-\frac{h}{2}\sum^{2n}_{j=1}f(t+jh/2-h/4) \label{eqT32}  \end{align}

\item
{\em The case $m=4$}:
\begin{align}
\widehat{T}^{(0)}_{4,n}[f]&=h\sum^{n-1}_{j=1}f(t+jh)-\frac{\pi^4}{45}g(t)h^{-3}
-\frac{\pi^2}{6}g''(t)h^{-1}+\frac{1}{24}g^{(4)}(t)h \label{eqT40} \\
\widehat{T}^{(1)}_{4,n}[f]&=h\sum^{n}_{j=1}f(t+jh-h/2)-\frac{\pi^4}{3}g(t)h^{-3}
-\frac{\pi^2}{2}g''(t)h^{-1} \label{eqT41} \\
\widehat{T}^{(2)}_{4,n}[f]&=2h\sum^{n}_{j=1}f(t+jh-h/2)
-\frac{h}{2}\sum^{2n}_{j=1}f(t+jh/2-h/4)+2\pi^4g(t)h^{-3} \label{eqT42}  \\
\widehat{T}^{(3)}_{4,n}[f]&=\frac{16h}{7}\sum^{n}_{j=1}f(t+jh-h/2)
-\frac{5h}{7}\sum^{2n}_{j=1}f(t+jh/2-h/4)\notag \\
&\hspace{4.5cm}+\frac{h}{28}\sum^{4n}_{j=1}f(t+jh/4-h/8) \label{eqT43}
\end{align}
\end{enumerate}

Each of the quadrature formulas  $\widehat{T}^{(s)}_{m,n}[f]$ above is obtained by performing $s$ steps of ``Richardson-like extrapolation'' on the sequence $\{\widehat{T}^{(0)}_{m,2^kn}[f]\}^s_{k=0}$.
Indeed, for $s=1$ (eliminating only the power $h^1$),  for $s=2$ (eliminating only the powers $h^1,h^{-1}$),  and for $s=3$ (eliminating only the powers $h^1,h^{-1},h^{-3}$), we have, respectively,
$$\widehat{T}^{(1)}_{m,n}[f]=2\widehat{T}^{(0)}_{m,2n}[f]-\widehat{T}^{(0)}_{m,n}[f],$$
\begin{align*}
\widehat{T}^{(2)}_{m,n}[f]&=2\widehat{T}^{(1)}_{m,n}[f]-\widehat{T}^{(1)}_{m,2n}[f]\\
&=-2\widehat{T}^{(0)}_{m,n}[f]+5\widehat{T}^{(0)}_{m,2n}[f]-2\widehat{T}^{(0)}_{m,4n}[f],
\end{align*}
and
\begin{align*} \widehat{T}^{(3)}_{m,n}[f]&=\frac{8}{7}\widehat{T}^{(2)}_{m,n}[f]-\frac{1}{7}\widehat{T}^{(2)}_{m,2n}[f]\\
&= \frac{16}{7}\widehat{T}^{(1)}_{m,n}[f]-\frac{10}{7}\widehat{T}^{(1)}_{m,2n}[f]+
\frac{2}{7}\widehat{T}^{(1)}_{m,4n}[f]\\
&=-\frac{16}{7}\widehat{T}^{(0)}_{m,n}[f]+6\widehat{T}^{(0)}_{m,2n}[f]
-3\widehat{T}^{(0)}_{m,4n}[f]+\frac{2}{7}\widehat{T}^{(0)}_{m,8n}[f].\end{align*}
In general, eliminating only the powers $h^1,h^{-1},h^{-3},\ldots, h^{-2s+3},$ we have
\be\label{eqalpha} \widehat{T}^{(s)}_{m,n}[f]=\sum^s_{k=0}\alpha^{(s)}_{m,k} \widehat{T}^{(0)}_{m,2^kn}[f],\quad
\sum^s_{k=0}\alpha^{(s)}_{m,k}=1;\quad \text{$\alpha^{(s)}_{m,k}$ independent of $n$}.\ee

\noindent{\bf Remarks:} \begin{enumerate}
\item
The quadrature formulas  $\widehat{T}^{(1)}_{1,n}[f]$ and $\widehat{T}^{(1)}_{2,n}[f]$  were derived and studied in \cite{Sidi:1988:QMP} and \cite{Sidi:2013:CNQ}, respectively.
\item
 In case $g^{(k)}(t)$, $k=1,2,\ldots,$  are not known or cannot be computed exactly, we can replace them wherever they are present in \eqref{eqT10}--\eqref{eqT42} by suitable approximations  based on the already computed (i)\,$g(t+jh)$ in case of
$\widehat{T}^{(0)}_{m,n}[f]$, and (ii)\,$g(t+jh-h/2)$ in case of $\widehat{T}^{(1)}_{m,n}[f]$, for example.
We can use differentiation formulas based on finite differences as approximations, for example. Of course, the error expansions of the quadrature formulas will now have additional powers of $h$ that result from the differentiation formulas used.
(For another approach that uses trigonometric interpolation and also  preserves spectral accuracy, see subsection \ref{sse63}.)
\item
In case $g(x)$ is not known, which happens when $f(x)$ is given as a black box, for example, or in case we do not wish to approximate the different $g^{(k)}(x)$,  quadrature formulas that do not involve $g(x)$ become very useful. The formulas $\widehat{T}^{(1)}_{1,n}[f]$ in \eqref{eqT11}, $\widehat{T}^{(2)}_{2,n}[f]$ in \eqref{eqT22},  $\widehat{T}^{(2)}_{3,n}[f]$ in \eqref{eqT32},
and $\widehat{T}^{(3)}_{4,n}[f]$ in \eqref{eqT43}  do not involve $g(x)$.
\end{enumerate}

\subsection{General convergence theorem}
We now state a  convergence theorem concerning all the quadrature formulas $\widehat{T}^{(s)}_{m,n}[f]$ defined in \eqref{eqalpha} in general, and those
in \eqref{eqT10}--\eqref{eqT42} in particular. This theorem
  results from the developments above, especially from the fact that the asymptotic expansions of $\widehat{T}^{(0)}_{m,n}[f]-I[f]$ as $h\to0$ are all empty:

  \begin{theorem}\label{th13} Let $f(x)$ be as in \eqref{eq1}--\eqref{eq2},
  and let the numerical quadrature formulas $\widehat{T}^{(s)}_{m,n}[f]$ be as defined above. Then $\lim_{n\to\infty}\widehat{T}^{(s)}_{m,n}[f]=I[f]$, and we have
 \be\label{eq34}
  \widehat{T}^{(s)}_{m,n}[f]-I[f]=o(n^{-\mu})\quad\text{as $n\to\infty$}\quad \forall \mu>0.\ee
In words, the errors in the $\widehat{T}^{(s)}_{m,n}[f]$ tend to zero as $n\to\infty$ faster than
 {\em every} negative power of $n$.
  \end{theorem}

\begin{proof}
We begin by  observing that, by \eqref{eqTtilde}, \eqref{evenm}--\eqref{oddm}, and
\eqref{eqthat0even}--\eqref{eqthat0odd}, there holds
\be \label{eqT0c} \widehat{T}^{(0)}_{m,n}[f]-I[f]=o(n^{-\mu})\quad\text{as $n\to\infty$}\quad \forall \mu>0,\ee
that is, \eqref{eq34} is true for $s=0$.
Next  by \eqref{eqalpha},
$$ \widehat{T}^{(s)}_{m,n}[f]-I[f]=\sum^s_{k=0}\alpha^{(s)}_{m,k} (\widehat{T}^{(0)}_{m,2^kn}[f]-I[f]).$$
Letting $n\to\infty$ and invoking \eqref{eqT0c}, the result in \eqref{eq34}  follows.
\end{proof}

\noindent{\bf Remarks:} \begin{enumerate}
\item
In the nomenclature  of the common literature, the quadrature formulas
$\widehat{T}^{(s)}_{m,n}[f]$ have {\em spectral} accuracy.
Thus,  $\widehat{T}^{(s)}_{m,n}[f]$  are excellent numerical quadrature formulas for computing $I[f]$ when $f(x)$ is
 infinitely differentiable and $T$-periodic on $\mathbb{R}_t$, with $\mathbb{R}_t$ as defined in \eqref{eq2}. This should be compared with most existing quadrature formulas based on trapezoidal sums, which  have errors that behave at  best like $O(n^{-\nu})$ for some low value of $\nu>0$.
\item
In case $f(z)$, the analytic continuation of $f(x)$ to the complex $z$-plane, is analytic in the strip $|\Im z|<\sigma$, the result of Theorem \ref{th13} can be improved optimally at least for $m=1,2,3.$ We now have that  the errors $\widehat{T}^{(s)}_{m,n}[f]-I[f]$, for every $s$,
tend to zero as $n\to\infty$ like $e^{-2n\pi \sigma/T}$ for all practical purposes,
as shown in \cite{Sidi:1988:QMP} for $m=1$, in \cite{Sidi:2013:CNQ} for $m=2$, and in
\cite{Sidi:2019:SSI-P2} for $m=3$.
\end{enumerate}

\subsection{Analysis of the  $\widehat{T}^{(s)}_{m,n}[f]$ in floating-point arithmetic}
\label{sse43}
Due to the fact that the integrand $f(x)$ tends to infinity as $x\to t$, the quadrature formulas $\widehat{T}^{(s)}_{m,n}[f]$ are likely to present some stability issues when applied in
floating-point (or finite-precision) arithmetic.
Before proceeding further, we would like to address this issue in some detail. We will study $\widehat{T}^{(0)}_{3,n}[f]$ only; the studies of   $\widehat{T}^{(s)}_{m,n}[f]$ with general $m$ and $s$  are similar and so are the conclusions derived from them.

Let us denote the numerically computed $\widehat{T}^{(0)}_{3,n}[f]$ by $\bar{T}^{(0)}_{3,n}[f]$. Then the true numerical  error is $(\bar{T}^{(0)}_{3,n}[f]-I[f])$, and we can rewrite it as
$$\bar{T}^{(0)}_{3,n}[f]-I[f]=\big(\bar{T}^{(0)}_{3,n}[f]-\widehat{T}^{(0)}_{3,n}[f]\big)+ \big(\widehat{T}^{(0)}_{3,n}[f]-I[f]\big), $$ and we can bound it as in
$$\big|\bar{T}^{(0)}_{3,n}[f]-I[f]\big|\leq\big|\bar{T}^{(0)}_{3,n}[f]-\widehat{T}^{(0)}_{3,n}[f]\big|+ \big|\widehat{T}^{(0)}_{3,n}[f]-I[f]\big|.$$
Clearly, the theoretical error  $(\widehat{T}{}^{(0)}_{3,n}[f]-I[f])$  tends to zero faster than any negative power of $n$ by Theorem \ref{th13}.  Therefore, we need to analyze
$(\bar{T}^{(0)}_{3,n}[f]-\widehat{T}^{(0)}_{3,n}[f])$, which is the source of
numerical instability.

For all practical purposes, it is clear from \eqref{eqT30} that the stability issue arises as a result of errors committed in computing $g(x)$ and its derivatives in the interval $[a,b]$
 because $f(x)$ is given and computed on the interval $[a,b]$ and $f(x)=f(x-T)$ for $x\in[b,b+T]$ since $f(x)$ is $T$-periodic.\footnote{Note that even a small error committed when computing  $g(x)$ is magnified by the denominator $(x-t)^3$ when $x$ is close to $t$.} Thus, with the integer $r$ being  such that $t+rh\leq b<t+(r+1)h$,  the sum $\sum^{n-1}_{j=1}f(t+jh)$ in \eqref{eqT30} is actually computed as
$$\sum^{r}_{j=1}f(t+jh)+\sum^{n-1}_{j=r+1}f(t+jh-T)=
\sum^{r}_{j=1}f(t+jh)+\sum_{j=-n+r+1}^{-1}f(t+jh)=
\sum^{r}_{\substack{j=-n+r+1 \\ j\neq0}}f(t+jh).$$
 We are assuming that the rest of the computations are being carried out with no errors.

 Now,  the computed $g(x)$, which we shall denote by $\bar{g}(x)$, is given as $\bar{g}(x)=g(x)[1+\eta(x)]$, where $\eta(x)$ is the relative error in $\bar{g}(x)$. Thus, letting $y_j=t+jh$, we have
$$\bar{T}^{(0)}_{3,n}[f]-\widehat{T}^{(0)}_{3,n}[f]=h\sum^{r}_{\substack{j=-n+r+1 \\ j\neq0}}\frac{g(y_j)\eta(y_j)}{(y_j-t)^3}
-\frac{\pi^2}{3}g'(t)\eta_1(t)h^{-1}+\frac{1}{6}g'''(t)\eta_3(t)h,$$
where we have denoted by $\eta_1(t)$ and $\eta_3(t)$ the relative errors in the computed $g'(t)$ and $g'''(t)$, respectively. Assuming that $g(x)$, $g'(x)$, and $g'''(x)$
are being computed with maximum precision allowed by the floating-point arithmetic being used, we have $\big|\eta(y_j)\big|\leq {\bf u}$, $\big|\eta_1(t)\big|\leq {\bf u}$, and $\big|\eta_3(t)\big|\leq {\bf u}$, where ${\bf u}$ is the roundoff unit of this arithmetic. Therefore,
\begin{align*} \big|\bar{T}^{(0)}_{3,n}[f]-\widehat{T}{}^{(0)}_{3,n}[f]\big|&\leq \|g\|{\bf u} h^{-2}
\sum^{r}_{\substack{j=-n+r+1 \\ j\neq0}}\frac{1}{|j|^3}+\frac{\pi^2}{3}\|g'\|{\bf u}h^{-1}+
\frac{1}{6}\|g'''\|{\bf u}h,\quad \\
&\leq \bigg(2\zeta(3)\|g\|+\frac{\pi^2}{3}\|g'\| h+\frac{1}{6}\|g'''\| h^3\bigg){\bf u}h^{-2}\\
&\leq K(n){\bf u}n^2,\quad  K(n)=\frac{2\zeta(3)}{T^2}\|g\|+\frac{\pi^2}{3Tn}\|g'\|
 +\frac{T}{6n^{3}}\|g'''\| .\end{align*}
Here $\|w\|=\max_{a\leq x\leq b}\big|w(x)\big|$ and $\zeta(3)=\sum^\infty_{k=1}k^{-3}$.

The conclusion from this is that $(\bar{T}^{(0)}_{3,n}[f]-\widehat{T}^{(0)}_{3,n}[f])$
 will dominate the true error $(\bar{T}^{(0)}_{3,n}[f]-I[f])$ for large $n$, depending on the size of {\bf u} (equivalently, whether we are using single- or double- or quadruple-precision arithmetic).  Fortunately, substantial accuracy will have been achieved by $\bar{T}^{(0)}_{3,n}[f]$ before $n$ becomes large since $(\widehat{T}^{(0)}_{3,n}[f]-I[f])$  tends to zero faster than $n^{-\mu}$ for {\em every} $\mu>0$. Tables \ref{ta0}--\ref{ta2} that result from the numerical example in the next section  amply substantiate this conclusion.

Finally, we would like to note that the abscissas of the formulas $\widehat{T}^{(s)}_{3,n}[f]$ will never be arbitrarily close to the point of singularity $x=t$; the smallest distance from this point is $h$, $h/2$, and $h/4$
for $s=0,1,2,$ respectively. This is not the case for most known formulas.

\section{A numerical example}\label{se5}
\setcounter{equation}{0} \setcounter{theorem}{0}
\setcounter{table}{0}
We can apply the quadrature formulas  $\widehat{T}^{(s)}_{m,n}$ we have derived to supersingular integrals   $ I[f]=\intBar^b_a f(x)\,dx$,  where $f(x)$ is $T$-periodic,  $T=b-a$, and is of the form
$$f(x)=\theta_m(x-t)u(x),\quad \theta_m(y)=\begin{cases}\displaystyle
\frac{\cos\frac{\pi y}{T}}{\sin^{2r-1}\frac{\pi y}{T}},\quad m=2r-1,\\
\displaystyle
\frac{1}{\sin^{2r}\frac{\pi y}{T}},\quad\quad m=2r, \end{cases}\quad r=1,2,\ldots.$$
 Such integrals arise from Cauchy transforms on the unit circle
$$J_m[w]=\intBar_\Gamma \frac{w(\zeta)}{(\zeta-z)^m}\,d\zeta,\quad z\in \Gamma=\{\zeta: |\zeta|=1\},\quad m=1,2,\ldots.$$
 Actually,
making the substitution $\zeta=e^{\mrm{i}x}$, $0\leq x\leq 2\pi$, so that $T=2\pi$, and letting $t\in[0,2\pi]$ be such that $z=e^{\mrm{i}t}$,  $J_m[w]$  becomes
$$ J_m[w]=\frac{\mrm{i}e^{\mrm{i}(1-m)t}}{(2\mrm{i})^m}\intBar^{2\pi}_0
\frac{\exp[\mrm{i}(2-m)\frac{x-t}{2}]}{\sin^m\frac{x-t}{2}} w(e^{\mrm{i}x})\,dx.
$$
After some manipulation, it can be shown that
$$ J_1[w]=\frac{1}{2}\intBar^{2\pi}_0\bigg(\frac{\cos\frac{x-t}{2}}{\sin\frac{x-t}{2}}+\mrm{i}\bigg)
w(e^{\mrm{i}x})\,dx,$$
$$J_2[w]=-\frac{\mrm{i}e^{-\mrm{i}t}}{4}\intBar^{2\pi}_0\frac{1}{\sin^2\frac{x-t}{2}}\,
w(e^{\mrm{i}x})\,dx,$$
$$J_3[w]=-\frac{e^{-\mrm{i}2t}}{8}\intBar^{2\pi}_0\bigg(\frac{\cos\frac{x-t}{2}}{\sin^3\frac{x-t}{2}}-
\mrm{i}\frac{1}{\sin^2\frac{x-t}{2}}\bigg)w(e^{\mrm{i}x})\,dx.$$
$$J_4[w]=\frac{\mrm{i}e^{-\mrm{i}3t}}{16}\intBar^{2\pi}_0\bigg(\frac{1}{\sin^4\frac{x-t}{2}}
-2\mrm{i}\frac{\cos\frac{x-t}{2}}{\sin^3\frac{x-t}{2}}
-2\frac{1}{\sin^2\frac{x-t}{2}}\bigg)w(e^{\mrm{i}x})\,dx.$$
For all $m\geq2$, we have
$$J_m[w]=\frac{\mrm{i}e^{\mrm{i}(1-m)t}}{(2\mrm{i})^m}\intBar^{2\pi}_0
\bigg[\sum^m_{k=2}\alpha_{m,k}\theta_k(x-t)\bigg]
w(e^{\mrm{i}x})\,dx,\quad\text{for some constants $\alpha_{m,k}$.}$$

We have applied  the quadrature formulas $\widehat{T}^{(s)}_{3,n}$ to supersingular integrals
$ I[f]=\intBar^b_a f(x)\,dx$,  where $f(x)$ is $T$-periodic and of the form
\be \label{eqex1} f(x)= \frac{\cos\frac{\pi(x-t)}{T}}{\sin^3\frac{\pi(x-t)}{T}}\,u(x),\quad
u(x)\ \ \text{$T$-periodic};\quad T=b-a. \ee
In order to approximate such integrals via the formulas $\widehat{T}{}^{(0)}_{3,n}[f]$, $\widehat{T}{}^{(1)}_{3,n}[f]$, and $\widehat{T}{}^{(2)}_{3,n}[f]$, we need to determine the  quantities $g'(t)$ and $g'''(t)$.  Now, $g(x)=(x-t)^3f(x)$ can be expressed as
$$g(x)=\bigg(\frac{T}{\pi}\bigg)^3 \frac{z^3\cos z}{\sin^3 z} u(x),\quad
z=\frac{\pi(x-t)}{T}.$$
Upon expanding in powers of $z$, we obtain
$$ \frac{z^3\cos z}{\sin^3 z}=1+O(z^4)\quad \text{as $z\to0$},$$
and, therefore,
\be \label{eqex2} g^{(i)}(t)=\bigg(\frac{T}{\pi}\bigg)^3u^{(i)}(t),\quad i=0,1,2,3.\ee

Unfortunately, we are not  aware of the existence of tables of supersingular periodic integrals when $f(x)$ is given as in \eqref{eqex1}. Therefore, we need to {\em construct} a simple but nontrivial periodic $u(x)$ for which $I[f]$ is given analytically and can  easily be computed. This is what we do next.

We apply the three quadrature formulas developed in Section \ref{se4}, with $T=2\pi$,  to
\be \label{eq555}I[f]=\intBar^\pi_{-\pi} f(x)\,dx, \quad f(x)=\frac{\cos\frac{x-t}{2}}{\sin^3\frac{x-t}{2}}\,u(x),\ee
with
\be \label{eq556} u(x)=\sum^\infty_{m=0}\eta^m \cos mx=\frac{1-\eta\cos x}
{1-2\eta\cos x +\eta^2}, \quad \text{$\eta$ real,} \quad \big|\eta\big|<1,\ee
which follows from
$$ u(x)=\Re \sum^\infty_{m=0}\eta^m e^{\mrm{i}mx}=
\Re\frac{1}{1-\eta e^{\mrm{i}x}}.$$
Clearly, $u(x)$ is $2\pi$-periodic, and so is $f(x)$. In addition, $u(x)$ is analytic in the strip $\big|\text{Im}\,z\big|<\sigma=\log \eta^{-1}.$

To obtain  an analytical expression for $I[f]$, we proceed as follows:

By the fact that $u(x)=\tfrac{1}{2}\sum^\infty_{m=0}\eta^m(e^{\mrm{i}mx}+e^{-\mrm{i}mx})$
and   by
$$ \intBar^\pi_{-\pi}\frac{\cos\frac{x-t}{2}}{\sin^3\frac{x-t}{2}}e^{\mrm{i}mx}\,dx=
-\text{sgn}(m)\,\mrm{i}4\pi m^2 e^{\mrm{i}mt},\quad m=0,\pm1,\pm2,\ldots,$$
which follows from Theorem 2.2  in Sidi \cite{Sidi:2019:SSI-P2}, we have
\begin{align}\label{eq557}I[f]&=
\frac{1}{2}\sum^\infty_{m=0}\eta^m(-\mrm{i}4\pi m^2 e^{\mrm{i}mt}+\mrm{i}4\pi m^2 e^{-\mrm{i}mt}) \notag\\
&=4\pi \Im \bigg[\bigg(\eta\frac{\partial}{\partial\eta}\bigg)^2\sum^\infty_{m=0}\eta^me^{\mrm{i}mt}\bigg] \notag\\
&= 4\pi \Im \bigg(\eta\frac{\partial}{\partial\eta}\bigg)^2\frac{1}{1-\eta e^{\mrm{i}t}} \notag\\
&=4\pi \Im \frac{\eta e^{\mrm{i}t}(1+\eta e^{\mrm{i}t})}{(1-\eta e^{\mrm{i}t})^3}.\end{align}

We have applied $\widehat{T}{}^{(s)}_{3,n}[f]$ with $t=1$ and $\eta=0.1(0.1)0.5.$ The results of this computation, using quadruple-precision arithmetic for which
${\bf u}=1.93\times 10^{-34}$ (approximately 34 decimal digits),
are given in Tables \ref{ta0}--\ref{ta2}.

\begin{table}[htb]
$$
\begin{array}{||r|c|c|c|c|c||} \hline
n&{E}_{n}(\eta=0.1)&{E}_{n}(\eta=0.2)&{E}_{n}(\eta=0.3)&{E}_{n}(\eta=0.4)&{E}_{n}(\eta=0.5)\\
\hline\hline
 10&  2.91D-10&  5.83D-07&  3.61D-05&  1.70D-04&  8.68D-03\\
   20&  1.87D-20&  2.19D-14&  4.69D-11&  1.07D-07&  2.10D-05\\
    30&  1.33D-30&  2.35D-21&  1.72D-15&  2.07D-11&  2.61D-08\\
    40&  1.30D-30&  6.34D-28&  1.54D-20&  2.46D-15&  2.27D-11\\
   50&  5.61D-30&  6.06D-30&  9.29D-26&  2.06D-19&  1.24D-14\\
    60&  9.19D-32&  7.74D-32&  8.14D-31&  9.19D-24&  1.39D-18\\
    70&  1.40D-29&  1.42D-29&  1.51D-29&  6.35D-28&  1.41D-20\\
    80&  2.21D-29&  2.16D-29&  2.21D-29&  2.21D-29&  2.17D-23\\
    90&  5.90D-29&  6.20D-29&  6.41D-29&  6.30D-29&  2.22D-26\\
   100&  1.04D-30&  1.73D-30&  2.83D-30&  6.98D-31&  1.81D-29\\
\hline
  \end{array}
$$
 \caption{\label{ta0} Numerical results for the integral in \eqref{eq555}--\eqref{eq557}
 with $t=1$ throughout.
 Here $E_n(\eta=c)=\big|\widehat{T}{}^{(0)}_{3,n}[f]-I[f]\big|$ for $\eta=c$. }
\end{table}

\begin{table}[htb]
$$
\begin{array}{||r|c|c|c|c|c||} \hline
n&{E}_{n}(\eta=0.1)&{E}_{n}(\eta=0.2)&{E}_{n}(\eta=0.3)&{E}_{n}(\eta=0.4)&{E}_{n}(\eta=0.5)\\
\hline\hline
   10&  2.91D-10&  5.83D-07&  3.61D-05&  1.70D-04&  8.72D-03\\
    20&  1.87D-20&  2.19D-14&  4.69D-11&  1.07D-07&  2.10D-05\\
    30&  7.80D-31&  2.35D-21&  1.72D-15&  2.07D-11&  2.61D-08\\
    40&  3.75D-29&  6.72D-28&  1.54D-20&  2.46D-15&  2.27D-11\\
    50&  3.34D-30&  2.64D-30&  9.29D-26&  2.06D-19&  1.24D-14\\
   60&  5.20D-30&  5.45D-30&  4.14D-30&  9.19D-24&  1.39D-18\\
    70&  1.20D-28&  1.21D-28&  1.28D-28&  5.28D-28&  1.41D-20\\
    80&  2.28D-29&  1.19D-29&  2.56D-29&  3.07D-29&  2.17D-23\\
    90&  1.13D-27&  1.18D-27&  1.18D-27&  1.17D-27&  2.33D-26\\
   100&  5.96D-28&  6.18D-28&  6.17D-28&  6.20D-28&  5.79D-28\\
\hline
  \end{array}
$$
 \caption{\label{ta1} Numerical results for the integral in \eqref{eq555}--\eqref{eq557}
 with $t=1$ throughout.
 Here $E_n(\eta=c)=\big|\widehat{T}{}^{(1)}_{3,n}[f]-I[f]\big|$ for $\eta=c$. }
\end{table}

\begin{table}[htb]
$$
\begin{array}{||r|c|c|c|c|c||} \hline
n&{E}_{n}(\eta=0.1)&{E}_{n}(\eta=0.2)&{E}_{n}(\eta=0.3)&{E}_{n}(\eta=0.4)&{E}_{n}(\eta=0.5)\\
\hline\hline
 10&  5.83D-10&  1.17D-06&  7.22D-05&  3.40D-04&  1.75D-02\\
  20&  3.73D-20&  4.38D-14&  9.37D-11&  2.14D-07&  4.19D-05\\
  30&  3.64D-30&  4.69D-21&  3.45D-15&  4.13D-11&  5.21D-08\\
  40&  9.78D-29&  1.36D-27&  3.09D-20&  4.93D-15&  4.54D-11\\
  50&  6.02D-28&  6.24D-28&  1.86D-25&  4.12D-19&  2.48D-14\\
  60&  1.59D-27&  1.65D-27&  1.67D-27&  1.84D-23&  2.77D-18\\
  70&  2.56D-28&  2.21D-28&  2.06D-28&  1.07D-27&  2.81D-20\\
  80&  3.83D-29&  1.32D-28&  9.14D-29&  1.19D-28&  4.35D-23\\
  90&  6.75D-27&  7.02D-27&  7.14D-27&  6.99D-27&  3.78D-26\\
 100&  1.44D-27&  1.47D-27&  1.47D-27&  1.49D-27&  1.37D-27\\
\hline
  \end{array}
$$
 \caption{\label{ta2} Numerical results for the integral in \eqref{eq555}--\eqref{eq557}
 with $t=1$ throughout.
 Here $E_n(\eta=c)=\big|\widehat{T}{}^{(2)}_{3,n}[f]-I[f]\big|$ for $\eta=c$. }
\end{table}

Judging from Tables \ref{ta0}--\ref{ta2}, we may conclude that,  all three quadrature formulas  $\widehat{T}{}^{(s)}_{3,n}[f]$ produce approximately the same accuracies. Actually, as shown in Sidi \cite[Theorem 5.2]{Sidi:2019:SSI-P2},
$E^{(s)}_n(\eta)=\big|\widehat{T}{}^{(s)}_{3,n}[f]-I[f]\big|=O(\eta^n)$ as $n\to\infty$  for  all three formulas; that is, all three formulas converge at the same rate as $n\to\infty$. The numerical results in Tables \ref{ta0}--\ref{ta2} are in agreement with this theoretical result as can be checked easily.

In subsection \ref{sse43}, we analyzed the true error in $\bar{T}^{(0)}_{3,n}[f]$, the computed  $\widehat{T}{}^{(0)}_{3,n}[f]$,   and concluded that
  $$\big|\bar{T}^{(0)}_{3,n}[f]-I[f]\big|
  \leq K(n){\bf u}n^2+o(n^{-\mu})\quad\text{as $n\to\infty$} \quad\forall \mu>0,$$
  with $K(n)$ bounded for all large $n$.
That is, the accuracy of $\bar{T}^{(0)}_{3,n}[f]$ increases quickly (and exponentially)  like $\eta^n$ up to a certain point where the term $K(n){\bf u}n^2$ increases to the point where it prevents
$\bar{T}^{(0)}_{3,n}[f]$ from  picking  up more correct significant digits. This takes place after
$\bar{T}^{(0)}_{3,n}[f]$ has achieved a very good accuracy in floating-point arithmetic, allowed by the size of ${\bf u}$.
 The numerical results in Tables \ref{ta0}--\ref{ta2}  demonstrate the validity of this argument amply.

\section{Application  to numerical solution of periodic  \\ supersingular integral equations}\label{se6}
\setcounter{equation}{0} \setcounter{theorem}{0}
\subsection{Preliminaries}

We now consider the application of the quadrature formulas $\widehat{T}^{(s)}_{3,n}$ to the numerical solution
of supersingular integral equations of the form
\be\label{eq4a} \lambda \phi(t)+\intBar^b_a K(t,x)\phi(x)\,dx =w(t),\quad t\in(a,b),\quad \lambda \ \text{scalar},\ee
such that, with  $T$, $\mathbb{R}$, and $\mathbb{R}_t$  as in \eqref{eq2}, and the following hold in addition:

\begin{enumerate}
\item
$K(t,x)$ is $T$-periodic in both $x$ and $t$, and is in $C^\infty(\mathbb{R}_t)$
 as a function of $x$, and is of the form
\be\label{eq5a}K(t,x)=\frac{U(t,x)}{(x-t)^3},\quad  U\in C^\infty([a,b]\times [a,b]).\ee
That is, as a function of $x$, $K(t,x)$ has  poles of order 3
at the points $x=t+kT$, $k=0,\pm1,\pm2,\ldots.$
\item
$w(t)$ is $T$-periodic in $t$ and is in $C^\infty(\mathbb{R})$.
\item
The solution $\phi(x)$ is
$T$-periodic in $x$ and is  in $C^\infty(\mathbb{R})$. (That $\phi\in C^\infty(\mathbb{R})$ under the conditions imposed on $K(t,x)$ and $w(t)$ can be argued heuristically, as was done in \cite[Introduction]{Sidi:1988:QMP}.)
\end{enumerate}

In some cases,  additional conditions are imposed on the solution to ensure uniqueness, which we will skip below. We now turn to the development of  numerical methods for solving
 \eqref{eq4a}.

\subsection{The ``simple'' approach}
Noting that   the quadrature formula $\widehat{T}{}^{(2)}_{3,n}[f]$ uses only function values $f(x_j)$, and no derivatives of $g(x)$, it is clearly very convenient to use, and we try this quadrature formula first.

Since $h$, $h/2$, and $h/4$ all feature in $\widehat{T}{}^{(2)}_{3,n}$, we proceed as follows:
For a given integer $n$, let
$\widehat{h} =  T/(4n)$, and $x_j = a +j\widehat{h} $, $j = 0,1,\ldots,4n,\ldots.$
Then $x_{4n}=b$ and $h=4\widehat{h}$ in $\widehat{T}{}^{(2)}_{3,n}$. Let $t$ be any one of the $x_j$, say $t=x_i$, $i\in\{1,2,\ldots,4n\}$, and
 approximate the integral $\intBar^b_a K(x_i, x) \phi(x)\, dx$ by the rule $\widehat{T}{}^{(2)}_{3,n}$, namely,
\begin{align*} \widehat{T}{}^{(2)}_{3,n}[K(x_i,\cdot)\phi]=2&\cdot 4\widehat{h} \sum^n_{j=1}K(x_i,x_i+4j\widehat{h}-2\widehat{h})\phi(x_i+4j\widehat{h}-2\widehat{h})\\
&-2\widehat{h} \sum^{2n}_{j=1} K(x_i,x_i+2j\widehat{h}-\widehat{h})\phi(x_i+2j\widehat{h}-\widehat{h}).\end{align*}
Finally,  noting that, for $k\leq 4n$,
$$f(x_i+k\widehat{h})=f(a+(i+k)\widehat{h})=\begin{cases} f(x_{i+k})&\quad\text{if $i+k\leq 4n$}\\
f(x_{i+k-4n})&\quad \text{if $i+k> 4n$}\end{cases}$$
when $f(x)$ is $T$-periodic,
and replacing the $\phi(x_j)$  by corresponding  approximations $\widehat{\phi}_j$,
and  recalling that everything here is $T$-periodic, [for example,
 $\phi(x_{j+4n})=\phi(x_j+T)=\phi(x_j)$ for all $j$,
and the same holds true for $K(t,x)$ and $w(x)$],
 we write down the following set of $4n$   equations for the $4n$ unknown  $\widehat{\phi}_j$:
\be\label{eq100} \lambda \widehat{\phi}_i+\widehat{h}\sum^{4n}_{j=1} \epsilon_{ij}K(x_i,x_j)\widehat{\phi}_j=w(x_i),
\quad i=1,\ldots,4n,\ee
where
\be \label{eq101}
\epsilon_{ij}=\begin{cases}8\quad &\text{if\ $\big|i-j-2\big|$\ divisible by $4$,}\\
-2\quad &\text{if\ $\big|i-j-1\big|$\ divisible by $2$,} \\
0\quad &\text{otherwise.}\end{cases}
\ee
Note that $\epsilon_{ii}=0$ for all $i$, which means that $K(x_i,x_i)$ is avoided.
The linear equations in \eqref{eq100} can be rewritten in the form
\begin{align}
&\sum^{2n}_{j=1}\widehat{K}_{ij}\widehat{\phi}_j=w(x_i),\quad i=1,\ldots,4n,\\
&\widehat{K}_{ij}=\epsilon_{ij}\widehat{h}K(x_i,x_j) +\lambda\delta_{ij}.
\label{eq103}\end{align}
Here $\delta_{ij}$ stands for the Kronecker delta.\\

\noindent{\bf Remark:} Note that if we were to use  either of the quadrature formulas $\widehat{T}{}^{(0)}_{3,n}[K(x_i,\cdot)\phi]$ or $\widehat{T}{}^{(1)}_{3,n}[K(x_i,\cdot)\phi]$ instead of $\widehat{T}{}^{(2)}_{3,n}[K(x_i,\cdot)\phi]$, we would have to know  the first and third  derivatives (with respect to $x$) of $U(x_i,x)\phi(x)$ at $x=x_i$, which implies that we must have knowledge of $\phi'(x)$, $\phi''(x)$, and $\phi'''(x)$. Of course, one may think that this is problematic since  $\phi(x)$ is the unknown function that we are trying to determine.
The quadrature formula  $\widehat{T}{}^{(2)}_{3,n}[K(x_i,\cdot)\phi]$ has no such problem since it relies only on integrand values.  We take up this issue in our next (``advanced'') approach.

\subsection{The ``advanced'' approach} \label{sse63}
In view of the fact that, for $m=1,2,3,$  all approximations $\widehat{T}{}^{(s)}_{m,n}[f]$ converge to $I[f]$ as $n\to\infty$ at the {\em same} rate when $f(z)$ is analytic and $T$-periodic in the strip $|\Im z|<\sigma$, we may want to keep the number of abscissas in  $\widehat{T}{}^{(s)}_{m,n}[f]$ to a minimum. We  can achieve this goal for $m=3$, for example,  by using $\widehat{T}{}^{(0)}_{3,n}[f]$, which requires only $n$ abscissas, unlike the $4n$ abscissas required by $\widehat{T}{}^{(2)}_{3,n}[f]$. We  apply this approach to $\intBar^b_a K(t,x)\phi(x)\,dx$ next.

With $f(t,x)=K(t,x)\phi(x)=U(t,x)\phi(x)/(x-t)^3$ and \eqref{eq5a},
we have
$$ f(t,x)=\frac{g(t,x)}{(x-t)^3},\quad g(t,x)=U(t,x)\phi(x).$$
Letting
$$ g_k(t,x)=\frac{\partial^k}{\partial x^k}g(t,x), \quad U_k(t,x)=\frac{\partial^k}{\partial x^k}U(t,x),$$
we  have
$$ g_k(t,x)=\sum^k_{p=0}\binom{k}{p}U_{k-p}(t,x)\phi^{(p)}(x),\quad k=0,1,2,\ldots, $$
where  $\phi^{(k)}(x)$ is the $k^{\text{th}}$ derivative of $\phi(x)$.
Therefore, by \eqref{eqT30}, we have
$$\widehat{T}^{(0)}_{3,n}[K(t,\cdot)\phi]=h\sum^{n-1}_{j=1}K(t,t+jh)\phi(t+jh)
-\frac{\pi^2}{3}g_1(t,t)h^{-1}+\frac{1}{6}g_3(t,t)h,$$
which, after some simple manipulation, can be written in the form
\be\label{eqcvb}  \widehat{T}^{(0)}_{3,n}[K(t,\cdot)\phi]=h\sum^{n-1}_{j=1}K(t,t+jh)\phi(t+jh)
+\sum^3_{k=0}A_k(t,h)\phi^{(k)}(t),\ee where
\be\label{eqcvb1}\begin{split}
 A_0(t,h)&=-\frac{\pi^2}{3}U_1(t,t)h^{-1}+\frac{1}{6}U_3(t,t)h, \\
A_1(t,h)&=-\frac{\pi^2}{3}U_0(t,t)h^{-1}+\frac{1}{2}U_2(t,t)h, \\
A_2(t,h)&=\frac{1}{2}U_1(t,t)h,  \\
 A_3(t,h)&=\frac{1}{6}U_0(t,t)h. \end{split}\ee
The  unknown quantities here  are $\phi^{(j)}(x)$, $j=0,1,2,3$.
We can take care of $\phi^{(j)}(x)$, $j=1,2,3,$ as follows:
We first construct the trigonometric interpolation polynomial
$Q_n(x)$ for $\phi(x)$  over the set of (equidistant) abscissas $\{x_0,x_1,\ldots,x_{n-1}\}$ already used for  constructing $\widehat{T}^{(0)}_{3,n}[K(t,\cdot)\phi]$; therefore, $Q_n(x_j)=\phi(x_j)$, $j=0,1,\ldots,n-1$.
Now,  since $\phi(x)$ is $T$-periodic and  infinitely differentiable on $\mathbb{R}$,   it is known that $Q_n(x)$ converges to $\phi(x)$ over $[a,b]$ with spectral accuracy. Similarly, for each $k$,  $Q_n^{(k)}(x)$,  the $k^{\text{th}}$ derivative of $Q_n(x)$, converges to $\phi^{(k)}(x)$ over $[a,b]$ with spectral accuracy and at the {\em same} rate. Now, with
$x_j=a+jT/n,$  $Q_n(x)$ is  of the form
$$ Q_n(x)=\sum^{n-1}_{j=0}D_n(x-x_j)\phi(x_j),\quad D_n(x_s-x_j)=\delta_{sj},$$ where
$$ D_n(y)=\frac{1}{n}\,\sin \frac{n\pi y}{T}\,\cot\frac{\pi y}{T},\quad
\text{when $n$ is an even integer.}\footnote{There is a similar result for odd $n$, which we omit. What is important here is the main idea.}$$  Taking $t\in \{x_0,x_1,\ldots,x_{n-1}\}$, we thus have
$$\phi(t)=Q_n(t);\quad \phi^{(k)}(t)\approx Q_n^{(k)}(t)=\sum^{n-1}_{j=0}D_n^{(k)}(t-x_j)\phi(x_j),\quad k=1,2,\ldots.$$
Letting   now $t=x_i$ and $\widehat{\phi}_i\approx \phi(x_i)$, we can replace the integral equation in \eqref{eq4a} by the following set of $n$ equations for the $n$ unknown  $\widehat{\phi}_i$:

\begin{align*} \lambda \widehat{\phi}_i&+h\sum^{n-1}_{\substack{j=0\\ j\neq i}} K(x_i,x_j)\widehat{\phi}_j +A_0(x_i,h)\widehat{\phi}_i \notag\\
&+\sum^{n-1}_{j=0}\bigg[\sum^3_{k=1}A_k(x_i,h)D_n^{(k)}(x_i-x_j)\bigg]\widehat{\phi}_j
=w(x_i),\quad i=0,1,\ldots,n-1,\end{align*}
where we have used the fact that

$$f(x_i+k{h})=f(a+(i+k){h})=\begin{cases} f(x_{i+k})&\quad\text{if $i+k\leq n-1$}\\
f(x_{i+k-n})&\quad \text{if $i+k\geq  n$}\end{cases}$$
when $f(x)$ is $T$-periodic and  $k\leq n-1$.
Finally, these equations can be rewritten in the form
\be
\sum^{n-1}_{j=0}\widehat{K}_{ij}\widehat{\phi}_j=w(x_i),\quad i=0,1,\ldots,n-1,\ee
\be \widehat{K}_{ij}=[\lambda+A_0(x_i,h)]\delta_{ij}+hK(x_i,x_j)(1-\delta_{ij})
+\sum^3_{k=1}A_k(x_i,h)D_n^{(k)}(x_i-x_j).\ee

We note that the idea of using trigonometric interpolation was introduced originally by Kress \cite{Kress:1995:NSH} in connection with the numerical solution of hypersingular integral equations. Needless to say, it can be used for all the singular integral equations with kernels having  singularities
of the form $(x-t)^{-m}$ with arbitrary integers $m\geq1$.


\begin{thebibliography}{10}

\bibitem{Davis:1984:MNI}
P.J. Davis and P.~Rabinowitz.
\newblock {\em {Methods of Numerical Integration}}.
\newblock Academic Press, New York, second edition, 1984.

\bibitem{Evans:1993:PNI}
G.~Evans.
\newblock {\em {Practical Numerical Integration}}.
\newblock Wiley, New York, 1993.

\bibitem{Huang:2013:AEE}
Jin Huang, Zhu Wang, and Rui Zhu.
\newblock Asymptotic error expansions for hypersingular integrals.
\newblock {\em Adv. Comput. Math.}, 38:257--279, 2013.

\bibitem{Kaya:1987:SIE}
A.C. Kaya and F.~Erdogan.
\newblock On the solution of integral equations with strongly singular kernels.
\newblock {\em Quart. Appl. Math.}, 45:105--122, 1987.

\bibitem{Kress:1995:NSH}
R.~Kress.
\newblock On the numerical solution of a hypersingular integral equation in
  scattering theory.
\newblock {\em J. Comp. Appl. Math.}, 61:345--360, 1995.

\bibitem{Krommer:1998:CI}
A.R. Krommer and C.W. Ueberhuber.
\newblock {\em {Computational Integration}}.
\newblock SIAM, Philadelphia, 1998.

\bibitem{Kythe:2005:HCM}
P.K. Kythe and M.R. Sch{\"{a}}ferkotter.
\newblock {\em {Handbook of Computational Methods for Integration}}.
\newblock Chapman \& Hall/CRC Press, New York, 2005.

\bibitem{Li:2010:NCR}
Buyang Li and Weiwei Sun.
\newblock {Newton-Cotes} for {Hadamard} finite-part integrals on an interval.
\newblock 30:1235--1255, 2010.

\bibitem{Li:2010:SUN}
Jin Li, Xiaoping Zhang, and Dehao Yu.
\newblock Superconvergence and ultraconvergence of {Newton-Cotes} rules for
  supersingular integrals.
\newblock {\em J. Comp. Appl. Math.}, 233:2841--2854, 2010.

\bibitem{Lifanov:2004:HIE}
I.K. Lifanov, L.N. Poltavskii, and G.M. Vainikko.
\newblock {\em Hypersingular Integral Equations and their Applications}.
\newblock CRC Press, New York, 2004.

\bibitem{Luke:1969:SFA1}
Y.L. Luke.
\newblock {\em The {Special} {Functions} and {Their} {Approximations}},
  volume~I.
\newblock Academic Press, New York, 1969.

\bibitem{Lyness:2005:AET}
J.N. Lyness and G.~Monegato.
\newblock Asymptotic expansions for two-dimensional hypersingular integrals.
\newblock {\em Numer. Math.}, 100:293--329, 2005.

\bibitem{Monegato:1994:NEH}
G.~Monegato.
\newblock Numerical evaluation of hypersingular integrals.
\newblock {\em J. Comp. Appl. Math.}, 50:9--31, 1994.

\bibitem{Monegato:2009:DPA}
G.~Monegato.
\newblock Definitions, properties and applications of finite-part integrals.
\newblock {\em J. Comp. Appl. Math.}, 229:425--439, 2009.

\bibitem{Navot:1961:EEM}
I.~Navot.
\newblock An extension of the {Euler--Maclaurin} summation formula to functions
  with a branch singularity.
\newblock {\em J. Math. and Phys.}, 40:271--276, 1961.

\bibitem{Navot:1962:FEE}
I.~Navot.
\newblock A further extension of the {Euler--Maclaurin} summation formula.
\newblock {\em J. Math. and Phys.}, 41:155--163, 1962.

\bibitem{Olver:2010:NIST}
F.W.J. Olver, D.W. Lozier, R.F. Boisvert, and {C.W. Clark, editors}.
\newblock {\em {NIST Handbook of Mathematical Functions}}.
\newblock Cambridge University Press, Cambridge, 2010.

\bibitem{Sidi:2003:PEM}
A.~Sidi.
\newblock {\em Practical Extrapolation Methods: Theory and Applications}.
\newblock Number~10 in Cambridge Monographs on Applied and Computational
  Mathematics. Cambridge University Press, Cambridge, 2003.

\bibitem{Sidi:2012:EME-P1}
A.~Sidi.
\newblock {Euler--Maclaurin} expansions for integrals with arbitrary algebraic
  endpoint singularities.
\newblock {\em Math. Comp.}, 81:2159--2173, 2012.

\bibitem{Sidi:2012:EME-P2}
A.~Sidi.
\newblock {Euler--Maclaurin} expansions for integrals with arbitrary
  algebraic-logarithmic endpoint singularities.
\newblock {\em Constr. Approx.}, 36:331--352, 2012.

\bibitem{Sidi:2013:CNQ}
A.~Sidi.
\newblock Compact numerical quadrature formulas for hypersingular integrals and
  integral equations.
\newblock {\em J. Sci. Comput.}, 54:145--176, 2013.

\bibitem{Sidi:2014:AES}
A.~Sidi.
\newblock Analysis of errors in some recent numerical quadrature formulas for
  periodic singular and hypersingular integrals via regularization.
\newblock {\em Appl. Numer. Math.}, 81:30--39, 2014.

\bibitem{Sidi:2014:RES}
A.~Sidi.
\newblock Richardson extrapolation on some recent numerical quadrature formulas
  for singular and hypersingular integrals and its study of stability.
\newblock {\em J. Sci. Comput.}, 60:141--159, 2014.

\bibitem{Sidi:2019:SSI-P2}
A.~Sidi.
\newblock Exactness and convergence properties of some recent numerical
  quadrature formulas for supersingular integrals of periodic functions.
\newblock Technical report, Computer Science Dept., Technion--Israel Institute
  of Technology, 2019.

\bibitem{Sidi:1988:QMP}
A.~Sidi and M.~Israeli.
\newblock Quadrature methods for periodic singular and weakly singular
  {Fredholm} integral equations.
\newblock {\em J. Sci. Comput.}, 3:201--231, 1988.
\newblock Originally appeared as Technical Report No. 384, Computer Science
  Dept., Technion--Israel Institute of Technology, (1985), and also as ICASE
  Report No. 86-50 (1986).

\bibitem{Wu:2010:SCM}
Jiming Wu, Zihuan Dai, and Xiaoping Zhang.
\newblock The superconvergence of the composite midpoint rule for the
  finite-part integral.
\newblock {\em J. Comp. Appl. Math.}, 233:1954--1968, 2010.

\bibitem{Wu:2008:SNC}
Jiming Wu and Weiwei Sun.
\newblock The superconvergence of {Newton--Cotes} rules for the {Hadamard}
  finite-part integral on an interval.
\newblock {\em Numer. Math.}, 109:143--165, 2008.

\bibitem{Zeng:2014:NCQ}
Guang Zeng, Li~Lei, and Jin Huang.
\newblock A new construction of quadrature formulas for {Cauchy} singular
  integral.
\newblock {\em J. Comput. Anal. Appl.}, 17:426--436, 2014.

\bibitem{Zhang:2009:SCS}
Xiaoping Zhang, Jiming Wu, and Dehao Yu.
\newblock Superconvergence of the composite {Simpson's} rule for a certain
  finite-part integral and its applications.
\newblock {\em J. Comp. Appl. Math.}, 223:598--613, 2009.

\end{thebibliography}

\end{document}